\documentclass[fleqn,reqno,11pt,a4paper,final]{amsart}

\usepackage[a4paper,left=30mm,right=30mm,top=30mm,bottom=30mm,marginpar=20mm]{geometry}
\usepackage{amsmath}
\usepackage{amssymb}
\usepackage{amsthm}
\usepackage{amscd}
\usepackage{mathtools}
\usepackage[ansinew]{inputenc}
\usepackage{cite}
\usepackage{bbm}
\usepackage{color}
\usepackage[english=american]{csquotes}
\usepackage[final]{graphicx}
\usepackage{hyperref}
\usepackage{calc}
\usepackage{mathptmx}
\usepackage{t1enc}

\graphicspath{{../Pictures/}}

\numberwithin{equation}{section}

\newtheoremstyle{thmlemcorr}{10pt}{10pt}{\itshape}{}{\bfseries}{.}{10pt}{{\thmname{#1}\thmnumber{ #2}\thmnote{ (#3)}}}
\newtheoremstyle{thmlemcorr*}{10pt}{10pt}{\itshape}{}{\bfseries}{.}\newline{{\thmname{#1}\thmnumber{ #2}\thmnote{ (#3)}}}
\newtheoremstyle{remexample}{10pt}{10pt}{}{}{\bfseries}{.}{10pt}{{\thmname{#1}\thmnumber{ #2}\thmnote{ (#3)}}}
\newtheoremstyle{ass}{10pt}{10pt}{}{}{\bfseries}{.}{10pt}{{\thmname{#1}\thmnumber{ A#2}\thmnote{ (#3)}}}

\theoremstyle{thmlemcorr}
\newtheorem{theorem}{Theorem}
\numberwithin{theorem}{section}
\newtheorem{lemma}[theorem]{Lemma}

\theoremstyle{thmlemcorr*}
\newtheorem{theorem*}{Theorem}
\newtheorem{lemma*}[theorem]{Lemma}
\newtheorem{corollary*}[theorem]{Corollary}
\newtheorem{proposition*}[theorem]{Proposition}
\newtheorem{problem*}[theorem]{Problem}
\newtheorem{conjecture*}[theorem]{Conjecture}
\newtheorem{definition*}[theorem]{Definition}

\theoremstyle{remexample}
\newtheorem{remark}[theorem]{Remark}

\theoremstyle{ass}

\newcommand{\Ccal}{\mathcal{C}}

\newcommand{\R}{\mathbb{R}}

\newcommand{\T}{\mathbb{T}}
\newcommand{\Tbb}{\mathbb{T}}
\newcommand{\norm}[1]{\|#1\|}
\newcommand{\eps}{\epsilon}

\newcommand{\intT}[1]{\int_{\T^d}  #1}
\newcommand{\dxdt}{ dx dt}
\newcommand{\dx}{ dx}

\DeclareMathOperator{\diverg}{div}
\DeclareMathOperator{\dist}{dist}
\DeclareMathOperator{\supp}{supp}

\renewcommand{\eps}{\varepsilon}
\renewcommand{\epsilon}{\varepsilon}
\renewcommand{\phi}{\varphi}

\begin{document}

\title{Conservation of energy for the Euler-Korteweg equations}
\author{Tomasz D\k{e}biec \and Piotr Gwiazda \and Agnieszka \'{S}wierczewska-Gwiazda \and Athanasios Tzavaras.
}
\address{\textit{Tomasz D\k{e}biec:} Institute of Applied Mathematics and Mechanics, University of Warsaw, Banacha 2, 02-097 Warszawa, Poland}
\email{t.debiec@mimuw.edu.pl}
\address{\textit{Piotr Gwiazda:} Institute of Mathematics, Polish Academy of Sciences, \'Sniadeckich 8, 00-656 Warszawa, Poland, and Institute of Applied Mathematics and Mechanics, University of Warsaw, Banacha 2, 02-097 Warszawa, Poland}
\email{pgwiazda@mimuw.edu.pl}
\address{\textit{Agnieszka \'{S}wierczewska-Gwiazda:} Institute of Applied Mathematics and Mechanics, University of Warsaw, Banacha 2, 02-097 Warszawa, Poland}
\email{aswiercz@mimuw.edu.pl}
\address{\textit{Athanasios Tzavaras:} Computer, Electrical, Mathematical Sciences and Engineering Division, King Abdullah University of Science and Technology (KAUST), Thuwal, Saudi Arabia}
\email{athanasios.tzavaras@kaust.edu.sa}

\maketitle

\begin{abstract}
In this article we study the principle of energy conservation for the Euler-Korteweg system. 
We formulate an Onsager-type sufficient regularity condition for weak solutions of the Euler-Korteweg system to conserve the total energy.
The result applies to the system of Quantum Hydrodynamics.
\end{abstract}

\section{Introduction}
It is known since the works of Scheffer \cite{scheffer} and Shnirelmann \cite{shnirel} that weak solutions of the incompressible Euler equations exhibit behaviour very different to that of classical solutions. These "wild solutions", as they are  called since the seminal works of DeLellis and Sz\'ekelyhidi \cite{DLS09, DLS10},  are often highly unphysical 
- for instance there is a lack of uniqueness and the principle of conservation of energy can be violated.\\
\indent Dissipative solutions of incompressible Euler have been extensively studied in relation to the seminal Onsager conjecture \cite{On1949}. It states that there is a threshold regularity, namely $\frac13$-H\"{o}lder continuity, above which kinetic energy must be conserved, and below which anomalous dissipation might occur. This conjecture has been recently fully resolved, with non-conservative solutions of class $\Ccal([0,T];\Ccal^{\frac13 -}(\T^3))$ constructed by Isett \cite{isett}. See also \cite{BucDeLSzV} and \cite{isett2} for further developments on the subject.
\\
\indent The positive direction of Onsager's conjecture has been settled already in the 1990's by Constantin et al. \cite{ConstETiti} (after a partial result of Eyink \cite{eyink}). 
The method of mollification and estimation of commutator errors was employed to prove that, if a weak solution of the incompressible Euler system belongs to $u\in L^3([0,T],B_3^{\alpha,\infty}(\T^3))\cap\Ccal([0,T],L^2(\T^3))$, then the energy $\norm{u}_{L^2(\T^3)}$ is conserved in time.
The method of proof as well as the observation that Besov spaces provide a suitable environment for this kind of problem were later used by several authors in the context of other systems of fluid dynamics: like inhomogeneous incompressible Euler and compressible Euler \cite{FGSGW}, incompressible and compressible Navier-Stokes (resp. \cite{DuRo}, \cite{LeSh} and~\cite{DrivasEyink},~\cite{Yu}),  incompressible magnetohydrodynamics~\cite{KangLee},~\cite{Cafetal}, and general systems of  first order conservation laws ~\cite{GMSG}.  Onsager's conjecture was recently studied for incompressible Euler equations in bounded domains, cf. \cite{BarTiti}.
An overview of these results can be found in \cite{DGSG}.\\
\indent In the present paper we adapt the strategy of Constantin et al. \cite{ConstETiti} and Feireisl et al. \cite{FGSGW} to obtain an Onsager-type sufficient condition on the regularity of weak solutions to the Euler-Korteweg equations so that they conserve the total energy. We consider the isothermal Euler-Korteweg system in the from
\begin{equation}\label{EKintro}
\begin{aligned}
\partial_t(\rho u)+\diverg(\rho u\otimes u)&=-\rho \nabla\left(h'(\rho)+\frac{\kappa'(\rho)}{2}|\nabla\rho|^{2}-\diverg(\kappa(\rho)\nabla\rho)\right),\\
\partial_t\rho+\diverg(\rho u)&=0,\\
\end{aligned}
\end{equation}
in the domain $(0,T)\times\T^d$ for some fixed time $T>0$, where $\T^d$ is the $d$-dimensional torus.
Here $\rho\geq0$ is the scalar density of a fluid, $u$ is its velocity, $h=h(\rho)$ is the energy density and $\kappa=\kappa(\rho)>0$ is the coefficient of capillarity. 
We place the assumption on the functions $h$ and $\kappa$:
\begin{equation}\label{regassumption}
h, \kappa \in \Ccal^3(\mathcal{T})   
\end{equation}
where, depending on the actual form of $h$ and $\kappa$, the set $\mathcal{T}$ can be chosen to be $[0,\infty)$ or $(0,\infty)$. For instance when $\kappa(\rho) = \frac{1}{\rho}$, as for the QHD system below, then $\mathcal{T}=(0,\infty)$ and we have to be away from vacuum.
\\
\indent
While the analysis of the above system dates back to the 19th century, when the mathematical theory of phase interfaces and capillary effects was introduced,
it still attracts much attention. A modern derivation of the system can be found in \cite{DunnSerrin}. 
Concerning smooth solutions: in \cite{BDD} and \cite{BDDJ} local-in-time well-posedness and stability of special solutions are analysed, respectively. 
A relative energy identity is developed in \cite{GLT}, exploiting the variational structure of the system, and is used to show that  solutions of  \eqref{EKintro}  
converge to smooth solutions of the compressible
Euler system (before shock formation) in the vanishing capillarity limit $\kappa \to 0$, see~\cite{GT}.\\
\\
\indent
The situation with weak solutions is much less understood. Most results concern the Quantum Hydrodynamics system, obtained from \eqref{EKintro}  
when $\kappa(\rho) = \frac{\epsilon_{0}^{2}}{4\rho}$, with $\eps_0$ denoting the Planck constant. This takes the form
\begin{equation}
\label{intro-qhd}
    \begin{aligned}
    \partial_t\rho + \diverg(\rho u) &= 0,\\
    \partial_t(\rho u) + \diverg(\rho u\otimes u) + \nabla p(\rho) &= \frac{\epsilon_{0}^{2}}{2}\rho\nabla\left(\frac{\Delta\sqrt{\rho}}{\sqrt{\rho}}\right).
    \end{aligned}
\end{equation}
The interesting connection between QHD and the Schroedinger equation is used in \cite{GM} to provide conservative weak solutions for the
special case of zero pressure,  $p(\rho) = 0$. Existence of weak solutions for a (relatively limited) class of pressure functions 
is provided in \cite{AntMarc2} and \cite{AntMarc}.  The existence of wild solutions is possible
for \eqref{EKintro}, as pointed out in the recent work Donatelli et al. \cite{DonFeiMar}, where the method of "convex integration" is adapted 
to show non-uniqueness in the class of dissipative global weak solutions.
 
 \medskip
 The possibility of both conservative and dissipative solutions
raises the issue of  studying the Onsager conjecture for the Euler-Korteweg system \eqref{EKintro}.
We use Besov spaces $B_p^{\alpha,\infty}(\Omega)$, with $1 \le p < \infty$, $0 < \alpha < 1$ (see section \ref{besov} for the definition) and 
prove the following theorem:
\begin{theorem}\label{EKTheorem}
Suppose that~\eqref{regassumption} holds. Let $(\rho,u)$ be a solution of~\eqref{EKintro} in the sense of distributions. Assume
\begin{equation}\label{besovhypo}
u \in (B_3^{\alpha,\infty}\cap L^\infty)((0,T)\times\Tbb^d),\hspace{0.3cm}\rho, \nabla\rho, \Delta\rho \in (B_3^{\beta,\infty}\cap L^\infty)((0,T)\times\Tbb^d),\hspace{0.3cm}
\end{equation}
where $1>\alpha\geq\beta> 0$ such that $\min(2\alpha+\beta,\alpha+2\beta)>1.$\\
Then the energy is locally conserved, i.e.
\begin{equation*}\label{EnergyEq}
\begin{aligned}
\int_0^T\int_{\T^d}&\left(\frac{1}{2}\rho|u|^{2} + h(\rho) +\frac12\kappa(\rho)|\nabla\rho|^{2}\right)\partial_t\phi\ \dxdt\\
&\hspace{-2.5cm}+\int_0^T\int_{\T^d}\left(\rho u\left(\frac12|u|^{2} + h'(\rho)+ \frac12\kappa'(\rho)|\nabla\rho|^2 -\diverg(\kappa(\rho)\nabla\rho)\right) + \kappa(\rho)\nabla\rho\diverg(\rho u)\right)\cdot\nabla\phi\ \dxdt = 0
\end{aligned}
\end{equation*}
holds for every $\phi\in\Ccal_c^1((0,T)\times\Tbb^d)$.
\end{theorem}

\begin{remark}
If in addition we assume  the following conditions on $u$ and $\rho$
\[
\lim\limits_{|\xi|,\tau\to 0}\frac{1}{\tau}\int_0^T\frac{1}{|\xi|}\intT|u(t+\tau,x+\xi)-u(t,x)|^3\dxdt = 0,
\]
\[
\lim\limits_{|\xi|,\tau\to 0}\frac{1}{\tau}\int_0^T\frac{1}{|\xi|}\intT|\rho(t+\tau,x+\xi)-\rho(t,x)|^3\dxdt = 0,
\]
then, as pointed out by Shvydkoy \cite{shvydkoy}, see also Duchon and Robert \cite{DuRo}, one can  allow for the case $\alpha=\beta=\frac13$. For details see e.g. Proposition~3 in~\cite{DuRo}.
\end{remark}

The short proof of the main theorem is presented in the following section: it is preceded by an outline of Besov spaces and their basic relevant properties,
some preliminary material on the structure of the Euler-Korteweg system, followed by he main part of the proof in
section \ref{sec:energy}.


\section{Proof of the main theorem}

\subsection{Besov Spaces}\label{besov}

Let $\Omega = (0,T)\times\Tbb^d$. The Besov space $B_p^{\alpha,\infty}(\Omega)$, with $1 \le p < \infty$, $0 < \alpha < 1$,  is the space of functions $w\in L^{p}$ for which the norm 
\begin{equation}\label{besovnorm}
\norm{w}_{B_p^{\alpha,\infty}(\Omega)}:=\norm{w}_{L^p(\Omega)}+\sup_{t>0}\left\{t^{-\alpha}\sup\limits_{|\xi|\leq t}\norm{w(\cdot+\xi)-w}_{L^p(\Omega\cap(\Omega-\xi))}\right\}
\end{equation}
is finite, cf.~\cite{BenShar}. In fact, we can replace the semi-norm in~\eqref{besovnorm} with the following one
\begin{equation}\label{besovseminorm}
\sup_{\xi\in\Omega}\left\{|\xi|^{-\alpha}\norm{w(\cdot+\xi)-w}_{L^p(\Omega\cap(\Omega-\xi))}\right\}.
\end{equation}
Indeed, if $\xi^*$ and $t^*$ realize the suprema in~\eqref{besovnorm} with $|\xi^*|<t^*$, then taking $|\xi^*|<t<t^*$ would contradict the supremality of $t^*$. Therefore neccesarily $\xi^*=t^*$, thus producing~\eqref{besovseminorm}. We choose to think of the Besov norm in terms of~\eqref{besovseminorm}, as it is more convienient for our purposes.

We observe that if $\alpha\geq\beta$, then there is an inclusion $B_p^{\alpha,\infty}(\Omega)\subset B_p^{\beta,\infty}(\Omega)$. Further we remark that the space $(B_p^{\alpha,\infty}\cap L^\infty)(\Omega)$ is a Banach algebra. For details we refer the reader to \cite{BenShar}.

\medskip 

Let $\eta\in C_c^\infty(\R^{d+1})$ be a standard mollification kernel and we denote 
\begin{equation*}
\eta^\eps(x)=\frac{1}{\eps^{d+1}}\eta\left(\frac{x}{\eps}\right), \quad  w^\epsilon=\eta^\epsilon*w \quad \text{and} \quad f^\eps(w) = f(w)*\eta^\eps.
\end{equation*}
Note that the function $w^\epsilon$ is well-defined on $\Omega^\epsilon=\{x\in\Omega: \dist(x,\partial\Omega)>\epsilon\}$. 
The following inequalities will be extensively used in the proof of the main theorem.
\begin{lemma}\label{lemma:besovgrad}
For any function $u\in B_p^{\alpha,\infty}(\Omega)$ we have
\begin{align}
\norm{u(\cdot+\xi)-u(\cdot)}_{L^p(\Omega\cap(\Omega-\xi))}&\leq |\xi|^\alpha\norm{u}_{B_p^{\alpha,\infty}(\Omega)}\label{besovshift}\\ 
\norm{u^\eps-u}_{L^p(\Omega)}&\leq \eps^\alpha\norm{u}_{B_p^{\alpha,\infty}(\Omega)} \label{besoveps}\\
\norm{\nabla u^\eps}_{L^p(\Omega)}&\leq C\eps^{\alpha-1}\norm{u}_{B_p^{\alpha,\infty}(\Omega)} \label{besovepsgradient}
\end{align}
\end{lemma}
\begin{proof}
Inequality~\eqref{besovshift} follows directly from the definition of the norm in the space $B_p^{\alpha,\infty}(\Omega)$. To show~\eqref{besoveps} we write
\begin{equation*}
\begin{aligned}
|u^\eps(x) - u(x)| &\leq\int_{\supp\eta^\eps}\eta^\eps(y)|u(x-y)-u(x)|\ d y\leq\left(\int_{\supp\eta^\eps}\eta^\eps(y)|u(x-y)-u(x)|^p\ d y\right)^{\frac{1}{p}}.     
\end{aligned}
\end{equation*}
Therefore, by virtue of Fubini and~\eqref{besovshift}
\begin{equation*}
\begin{aligned}
\int_{\Omega}|u^\eps(x) - u(x)|^p\ \dx &\leq\int_{\supp\eta^\eps}\eta^\eps(y)\int_{\Omega}|u(x-y)-u(x)|^p\ \dx\ d y\\
&\leq \int_{\supp\eta^\eps}\eta^\eps(y)|y|^{p\alpha}\norm{u}^p_{B_p^{\alpha,\infty}(\Omega)}\ dy \leq \eps^{p\alpha}\norm{u}^p_{B_p^{\alpha,\infty}(\Omega)}.
\end{aligned}
\end{equation*}
For the last of the claimed inequalities we consider the convolution $\nabla u^\eps = \nabla\eta^\eps*u$ as a bounded linear operator $T:L^p(\Omega)\to L^p(\Omega)$. Then
\[
\norm{Tu}_{L^p}\leq C\eps^{-1}\norm{u}_{L^p}.
\]
On the other hand, writing $\nabla u^\eps = \eta^\eps*\nabla u$, we can think of $T$ as mapping $W^{1,p}(\Omega)$ into $L^p(\Omega)$. It then has unit norm.

Therefore, as the Besov space $B_p^{\alpha,\infty}$ is an interpolation space of exponent $\alpha$ for $L^p$ and $W^{1,p}$ (cf.~\cite[Corollary~4.13]{BenShar}), $T$ is bounded as an operator $B_p^{\alpha,\infty}(\Omega)\to L^p(\Omega)$ with
\[
\norm{Tu}_{L^p}\leq C\eps^{-(1-\alpha)}\norm{u}_{B_p^{\alpha,\infty}}.
\]
\end{proof}

\begin{lemma}\label{lemma:compositiongrad}
Let $v\in B_p^{\alpha,\infty}(\Omega,\R^m)$.
Suppose $f:\R^m\to\R$ is a $C^1$ function with $\frac{\partial f}{\partial v_i}\in L^\infty$ for each $i=1,\dots,m$.
Then 
\[
\norm{\nabla f(v^\eps)}_{L^p}\leq C\eps^{\alpha-1}\norm{v}_{B_p^{\alpha,\infty}}
\]
\end{lemma}
\begin{proof}
Since $\nabla f(v^\eps) = \sum\limits_{i=1}^m\frac{\partial f}{\partial v_i}(v^\eps)\nabla v^\eps_i$,
we have
\begin{equation*}
\norm{\nabla f(v^\eps)}_{L^p}\leq\sum\limits_{i=1}^m\norm{\frac{\partial f}{\partial v_i}(v^\eps)}_{L^\infty}\norm{\nabla v^\eps_i}_{L^p}\leq\max_{1\leq i\leq m}\norm{\frac{\partial f}{\partial v_i}}_{L^\infty}\sum\limits_{i=1}^m\norm{\nabla v^\eps_i}_{L^p}\leq C\eps^{\alpha-1}\sum\limits_{i=1}^m\norm{v_i}_{B_p^{\alpha,\infty}}
\end{equation*}
where the last inequality follows from Lemma~\ref{lemma:besovgrad}.
\end{proof}

\medskip

\subsection{Preliminaries}
System~\eqref{EKintro} can be written in conservative form
\begin{equation}\label{EKintroConservative}
\begin{aligned}
\partial_t(\rho u)+\diverg(\rho u\otimes u)&=\diverg \mathbb{S},\\
\partial_t\rho+\diverg(\rho u)&=0,\\
\end{aligned}
\end{equation}
where $\mathbb{S}$ is the Korteweg stress tensor
\[
\mathbb{S} = \left( -p(\rho) - \frac{\rho\kappa'(\rho)+\kappa(\rho)}{2}|\nabla\rho|^{2} + \diverg(\rho\kappa(\rho)\nabla\rho) \right)\mathbb{I} 
- \kappa(\rho)\nabla\rho\otimes\nabla\rho
\]
with $\mathbb{I}$ denoting the $d$-dimensional identity matrix and the local pressure defined as
\[ p(\rho) = \rho h'(\rho) - h(\rho).
\]
It is routine to show that a strong solution $(\rho, u)$ of the above system will satisfy the following local balance of total (kinetic and internal) energy
\begin{equation}\label{EnergyEq}
\begin{aligned}
\partial_t&\left(\frac{1}{2}\rho|u|^{2} + h(\rho) +\frac12\kappa(\rho)|\nabla\rho|^{2}\right)\\
&\hspace{-1cm}+ \diverg\left(\rho u\left(\frac12|u|^{2} + h'(\rho)+ \frac12\kappa'(\rho)|\nabla\rho|^2 -\diverg(\kappa(\rho)\nabla\rho)\right) + \kappa(\rho)\nabla\rho\diverg(\rho u)\right) = 0.
\end{aligned}
\end{equation}
Theorem~\ref{EKTheorem} gives sufficient conditions for regularity of weak solutions so that they obey the above energy equality in the sense of distributions.
To prove the theorem we employ the strategy of \cite{ConstETiti}, which was used in many works in the subject, including \cite{FGSGW} and \cite{GMSG}, where variants of the following lemma are an important ingredient.

\begin{lemma}\label{CommutatorEstimates}
Let $1\leq q<\infty$ and suppose $v\in L^{2q}((0,T)\times\T^d;\R^k)$ and $f\in \Ccal^2(\R^k,\R^N)$. If
\[\sup\limits_{i,j}\norm{\frac{\partial^2 f}{\partial v_i\partial v_j}}_{L^\infty} < \infty,
\]
then there exists a constant $C>0$ such that
\begin{equation}\label{quadestimate}
    \norm{f(v^\eps)-f^\eps(v)}_{L^q}\leq C\left(\norm{v^\eps-v}^2_{L^{2q}} + \sup\limits_{(s,y)\in\supp\eta_\eps}\norm{v(\cdot,\cdot)-v(\cdot-s,\cdot-y)}^2_{L^{2q}}\right).
\end{equation}
\end{lemma}
\begin{proof}
We observe that by Taylor's theorem we have
\begin{equation}\label{Taylor1}
\left|f(v^\eps(t,x))-f(v(t,x)))-Df(v(t,x))(v^\eps(t,x)-v(t,x))\right|\leq C|v^\eps(t,x)-v(t,x)|^2
\end{equation}
where the constant $C$ does not depend on the choice of $x$ and $t$. Similarly
\begin{equation}\label{Taylor2}
\left|f(v(s,y))-f(v(t,x))-Df(v(t,x))(v(s,y)-v(t,x))\right|\leq C|v(s,y)-v(t,x)|^2.
\end{equation}
Mollification of the last inequality with respect to $(s,y)$ yields, by virtue of Jensen's inequality
\begin{equation}\label{Taylor3}
\left|f^\eps(v(t,x))-f(v(t,x)-Df(v(t,x))(v^\eps(t,x)-v(t,x))\right|\leq C|v(\cdot,\cdot)-v(t,x)|^2*_{(s,y)}\eta^\eps.
\end{equation}
Combining ~\eqref{Taylor1} and ~\eqref{Taylor3} and using the triangle inequality we deduce the estimate
\begin{equation}\label{nonlinearestimate}
\left|f(v^\eps(t,x))-f^\eps(v(t,x))\right|\leq C\left(|v^\eps(t,x)-v(t,x)|^2 + |v(\cdot,\cdot)-v(t,x)|^2*_{(s,y)}\eta^\eps\right).
\end{equation}
Finally, we observe that
\begin{equation*}
    \begin{aligned}
    \int_{(0,T)\times\T^d}&\left||v(\cdot,\cdot)-v(t,x)|^2*_{(s,y)}\eta^\eps\right|^q\ \dxdt \\
    &\leq \int_{\supp\eta^\eps}\eta^\eps(s,y)\int_{(0,T)\times\T^d}|v(t-s,x-y)-v(t,x)|^{2q}\ \dxdt \ d y d s\\
    &\leq \sup\limits_{(s,y)\in\supp\eta_\eps}\norm{v(\cdot,\cdot)-v(\cdot-s,\cdot-y)}^{2q}_{L^{2q}}.
    \end{aligned}
\end{equation*}
\end{proof}

\subsection{Energy equality}\label{sec:energy}
We begin the proof of the theorem by mollifying the momentum equation in both space and time with kernel and notation as in section~\ref{besov} to obtain
\begin{equation}\label{smoothMomentum}
\partial_t(\rho u)^\eps + \diverg(\rho u\otimes u)^\eps = -\nabla p^\eps(\rho) + \diverg{S^\eps(\rho,\nabla\rho, \Delta \rho)},
\end{equation}
where
\begin{equation}
\begin{split}
S(\rho,q,r)& =   \left( - \frac12( \rho \kappa'(\rho)+\kappa(\rho))q^2+\diverg(\rho\kappa(\rho)q)\right)\mathbb{I} - \kappa(\rho)q\otimes q\\
&= \left(  \frac12( \rho \kappa'(\rho)+\kappa(\rho))q^2+\rho \kappa(\rho)r\right)\mathbb{I} - \kappa(\rho)q\otimes q
\end{split}
\end{equation}
Equation~\eqref{smoothMomentum} can be rewritten in terms of appropriate commutators to give
\begin{equation}\label{smoothEK}
\begin{aligned}
\partial_t&(\rho^\eps u^\eps)+\diverg((\rho u)^\eps\otimes u^\eps)+\nabla p(\rho^{\eps})-\diverg(S(\rho^\eps,\nabla\rho^\eps,\Delta \rho^\eps))\\
&=\partial_t(\rho^\eps u^\eps - (\rho u)^\eps) + \diverg((\rho u)^\eps\otimes u^\eps - (\rho u\otimes u)^\eps) + \nabla\left(p(\rho^\eps) - p^\eps(\rho)\right)\\ &\hspace{0.5cm}-\diverg{\left(S(\rho^\eps,\nabla\rho^\eps, \Delta \rho^\eps) - S^\eps(\rho,\nabla\rho, \Delta \rho)\right)}.
\end{aligned}
\end{equation}
We observe the following identities
\begin{equation*}
    \diverg((\rho u)^\eps\otimes u^\eps) = u^\eps\diverg{(\rho u)^\eps} + ((\rho u)^\eps\cdot\nabla)u^\eps
\end{equation*}
and
\begin{equation*}
-\rho^\epsilon\;\nabla\left(\frac12\kappa'(\rho^\eps)|\nabla\rho^\eps|^2-\diverg(\kappa(\rho^\eps)\nabla\rho^\eps)\right) = \diverg S(\rho^\eps,\nabla\rho^\eps, \Delta \rho^\eps).
\end{equation*}
Thus the left-hand side of equation~\eqref{smoothEK} can be written as
\begin{equation*}
\begin{aligned}
(\partial_t\rho)u^\eps &+ \rho^\eps\partial_t u^\eps + u^\eps\diverg(\rho u)^\eps + ((\rho u)^\eps\cdot\nabla)u^\eps\\
&+ \rho^\epsilon\;\nabla\left(h'(\rho^\eps)-\frac12\kappa'(\rho^\eps)|\nabla\rho^\eps|^2-\kappa(\rho^\eps)\Delta\rho^\eps\right).
\end{aligned}
\end{equation*}
Hence, upon multiplying with $u^\eps$, equation~\eqref{smoothEK} becomes
\begin{equation}\label{smoothEK2}
\begin{aligned}
\rho^\eps&\partial_t\left(\frac12|u^\eps|^2\right)+\left((\rho u)^\eps\cdot\nabla\right)\frac12|u^\eps|^2 + \rho^\eps u^\eps\;\nabla\left(h'(\rho^\eps)-\frac12\kappa'(\rho^\eps)|\nabla\rho^\eps|^2-\kappa(\rho^\eps)\Delta\rho^\eps\right)\\
&=r_1^\eps + r_2^\eps + r_3^\eps + r_4^\eps,
\end{aligned}
\end{equation}
where
\begin{equation*}
    \begin{aligned}
    r_1^\eps &= \partial_t(\rho^\eps u^\eps - (\rho u)^\eps)\cdot u^\eps,\\
    r_2^\eps &= \diverg((\rho u)^\eps\otimes u^\eps - (\rho u\otimes u)^\eps)\cdot u^\eps,\\
    r_3^\eps &= \nabla\left(p(\rho^\eps) - p^\eps(\rho)\right)\cdot u^\eps,\\
    r_4^\eps &= -\diverg{\left(S(\rho^\eps,\nabla\rho^\eps, \Delta \rho^\eps) - S^\eps(\rho,\nabla\rho, \Delta \rho)\right)}\cdot u^\eps.
    \end{aligned}
\end{equation*}
Using the mollified continuity equation
\begin{equation}\label{eq:smoothMass}
\partial_t\rho^\eps+\diverg(\rho u)^\eps=0,
\end{equation}
we can write the first two terms of~\eqref{smoothEK2} as
\begin{equation}\label{eq:eulerterms}
\begin{aligned}
    \rho^\eps&\partial_t\left(\frac12|u^\eps|^2\right)+\left((\rho u)^\eps\cdot\nabla\right)\frac12|u^\eps|^2 + \left(\partial_t\rho^\eps+\diverg(\rho u)^\eps\right)\frac12|u^\eps|^2 \\
    &=\partial_t\left(\frac12\rho^\eps|u^\eps|^2\right) + \diverg\left((\rho u)^\eps\;\frac12|u^\eps|^2\right).
\end{aligned}
\end{equation}
Combining equations~\eqref{smoothEK2} and~\eqref{eq:eulerterms} we obtain
\begin{equation}\label{smoothEK3}
\begin{aligned}
\partial_t&\left(\frac12\rho^\eps|u^\eps|^2\right) + \diverg\left((\rho u)^\eps\;\frac12|u^\eps|^2\right) + \rho^\eps u^\eps\;\nabla\left(h'(\rho^\eps)-\frac12\kappa'(\rho^\eps)|\nabla\rho^\eps|^2-\kappa(\rho^\eps)\Delta\rho^\eps\right)\\
&=r_1^\eps + r_2^\eps + r_3^\eps + r_4^\eps.
\end{aligned}
\end{equation}
We now rewrite the mollified continuity equation~\eqref{eq:smoothMass} in the form
\begin{equation*}
\partial_t\rho^\epsilon+\diverg(\rho^\epsilon u^\epsilon)=\diverg(\rho^\epsilon u^\epsilon-(\rho u)^\epsilon).
\end{equation*}
After multiplying this equation with
\[
h'(\rho^\eps) - \frac12\kappa'(\rho^\eps)|\nabla\rho^\eps|^2 - \kappa(\rho^\eps)\Delta\rho^\eps
\]
and rearranging, we obtain
\begin{equation}\label{smoothMass2}
\begin{aligned}
&\partial_t\left(h(\rho^\eps) + \frac12\kappa'(\rho^\eps)|\nabla\rho^\eps|^2\right) - \diverg\left(\kappa(\rho^\eps)\nabla\rho^\eps\partial_t\rho^\eps\right)\\
&\hspace{0.5cm}+\diverg(\rho^\eps u^\eps)\left(h'(\rho^\eps) - \frac12\kappa'(\rho^\eps)|\nabla\rho^\eps|^2 - \kappa(\rho^\eps)\Delta\rho^\eps\right)\\
&=r_5^\eps + r_6^\eps + r_7^\eps,
\end{aligned}
\end{equation}
where
\begin{equation*}
    \begin{aligned}
    r_5^\eps &= \diverg(\rho^\eps u^\eps - (\rho u)^\eps)\;h'(\rho^\eps),\\
    r_6^\eps &= -\diverg(\rho^\eps u^\eps - (\rho u)^\eps)\;\frac12\kappa'(\rho^\eps)|\nabla\rho^\eps|^2,\\
    r_7^\eps &= -\diverg(\rho^\eps u^\eps - (\rho u)^\eps)\kappa(\rho^\eps)\Delta\rho^\eps.
    \end{aligned}
\end{equation*}
Combining equations~\eqref{smoothEK3} and~\eqref{smoothMass2} we obtain
\begin{equation}\label{smoothEnergy}
\begin{aligned}
\partial_t&\left(\frac12\rho^\eps |u^\eps|^2 + h(\rho^\eps) + \frac12\kappa(\rho^\eps)|\nabla\rho^\eps|^2\right) + \diverg{\left((\rho u)^\eps\;\frac12|u^\eps|^2\right)}\\
&+\diverg\left(\rho^\eps u^\eps\left(h'(\rho^\eps) - \frac12\kappa'(\rho^\eps)|\nabla\rho^\eps|^2 - \kappa(\rho^\eps)\Delta\rho^\eps\right) + \kappa(\rho^\eps)\nabla\rho^\eps\diverg(\rho^\eps u^\eps) \right)\\
&\hspace{-0.5cm}=r_1^\eps + r_2^\eps + r_3^\eps + r_4^\eps + r_5^\eps + r_6^\eps + r_7^\eps.
\end{aligned}
\end{equation}
It follows that to prove the theorem it is sufficient to show that each commutator error term converges to zero in the distributional sense on $(0,T)\times\T^d$ as $\epsilon\to 0$.

\subsection{Commutator Estimates}
Let $\phi\in \Ccal_{c}^{1}((0,T)\times\mathbb{T}^{d})$ and take $\epsilon>0$ small enough so that $\supp\phi\subset (\epsilon,T-\epsilon)\times\Tbb^d$. We will show that for each $1\leq i\leq 7$ we have
\begin{equation*}
R_i^\eps\coloneqq\int_0^T\intT r_i^\eps\phi\ \dxdt \xrightarrow{\eps\to 0^+} 0.
\end{equation*}
The terms $R_1^\eps$ and $R_2^\eps$ are dealt with in the same way as in~\cite{FGSGW}. We recall these estimates for the reader's convenience.
For $R_1^\epsilon$ we observe that
\begin{equation}\label{pointwisedecomp}
\begin{aligned}
\rho^\eps u^\eps-(\rho u)^\eps&= (\rho^\eps-\rho)(u^\eps-u)\\
&\hspace{0.3cm}-\int_{-\eps}^\eps\intT\eta^\eps(\tau,\xi)(\rho(t-\tau,x-\xi)-\rho(t,x))(u(t-\tau,x-\xi)-u(t,x)) d\xi d\tau.
\end{aligned}
\end{equation}
The first part of $R_1^\epsilon$ therefore can be estimated by virtue of an integration by parts, H\"{o}lder inequality and estimates~\eqref{besoveps} and~\eqref{besovepsgradient} as
\begin{equation*}
\begin{aligned}
&\left|\int_0^T\intT\phi\partial_t\left((\rho^\eps-\rho)(u^\eps-u)\right)\cdot u^\eps\ \dxdt\right|\\
&\hspace{0.5cm}\leq\int_0^T\intT |(\rho^\eps-\rho)(u^\eps-u)|(|\partial_t\phi\; u^\eps| + |\phi\partial_t u^\eps|)\ \dxdt \\
&\hspace{0.5cm}\leq\norm{\phi}_{\Ccal^1}\norm{\rho^\eps-\rho}_{L^3}\norm{u^\eps-u}_{L^3}\norm{u^\eps}_{L^3} + \norm{\phi}_{\Ccal^0}\norm{\rho^\eps-\rho}_{L^3}\norm{u^\eps-u}_{L^3}\norm{\nabla u^\eps}_{L^3}\\
&\hspace{0.5cm}\leq C\eps^\beta\eps^\alpha\norm{\rho}_{B_3^{\beta,\infty}}\norm{u}^2_{B_3^{\alpha,\infty}}+ C\eps^\beta\eps^\alpha\eps^{\alpha-1}\norm{\rho}_{B_3^{\beta,\infty}}\norm{u}^2_{B_3^{\alpha,\infty}}
\end{aligned}
\end{equation*}
For the second part of $R_1^\epsilon$ according to~\eqref{pointwisedecomp}, we estimate (using integration by parts, Fubini,~\eqref{besovshift} and~\eqref{besovepsgradient})
\begin{equation*}
\begin{aligned}
&\left|\int_0^T\intT\phi\partial_t\int_{-\eps}^\eps\intT\eta^\eps(\tau,\xi)(\rho(t-\tau,x-\xi)-\rho(t,x))(u(t-\tau,x-\xi)-u(t,x)) d\xi d\tau\cdot u^\epsilon \dxdt\right|\\
&\leq C\norm{\phi}_{\Ccal^1}\eps^\beta\eps^\alpha\norm{\rho}_{B_3^{\beta,\infty}}\norm{u}^2_{B_3^{\alpha,\infty}}+C\norm{\phi}_{\Ccal^0}\eps^\beta\eps^\alpha\eps^{\alpha-1}\norm{\rho}_{B_3^{\beta,\infty}}\norm{u}^2_{B_3^{\alpha,\infty}}.
\end{aligned}
\end{equation*}
A similar estimation can be carried out for $R_2^\epsilon$. We write
\begin{equation*}
\begin{aligned}
(\rho u)^\eps\otimes u^\eps&-(\rho u\otimes u)^\eps=((\rho u)^\eps-\rho u)\otimes(u^\eps-u)\\
&-\int_{-\eps}^\eps\intT\eta^\eps(\tau,\xi)(\rho u(t-\tau,x-\xi)-\rho u(t,x))\otimes(u(t-\tau,x-\xi)-u(t,x))d\xi d\tau.
\end{aligned}
\end{equation*}
We observe that since $\alpha\geq\beta$ and the space $B_3^{\beta,\infty}\cap L^\infty$ is an algebra, we have $\rho u\in (B_3^{\beta,\infty}\cap L^\infty)((0,T)\times\T^d)$.
Thus the first part of $R_2^\epsilon$ can be estimated as
\begin{equation*}
\begin{aligned}
&\left|\int_0^T\intT\diverg\left(((\rho u)^\eps-\rho u)\otimes(u^\eps-u)\right)\cdot \phi u^\eps\ \dxdt\right|\\
&\leq\norm{\phi}_{C^1}\norm{(\rho u)^\eps-\rho u}_{L^3}\norm{u^\eps-u}_{L^3}\norm{u^\eps}_{L^3}+\norm{\phi}_{C^0}\norm{(\rho u)^\eps-\rho u}_{L^3}\norm{u^\eps-u}_{L^3}\norm{\nabla u^\eps}_{L^3}
\\
&\leq C\eps^\beta\eps^\alpha\norm{\rho}_{B_3^{\beta,\infty}}\norm{u}^2_{B_3^{\alpha,\infty}} + C\eps^\beta\eps^\alpha\eps^{\alpha-1}\norm{\rho u}_{B_3^{\beta,\infty}}\norm{u}^2_{B_3^{\alpha,\infty}}.
\end{aligned}
\end{equation*}
Likewise, for the second part of $R_2^\epsilon$ we get
\begin{equation*}
\begin{aligned}
&\left|\int_0^T\intT\diverg\left\{\int_{-\eps}^\eps\intT\eta^\eps(\tau,\xi)(\rho u(t-\tau,x-\xi)-\rho u(t,x))\otimes(u(t-\tau,x-\xi)-u(t,x))d\xi d\tau\right\}\cdot \phi u^\eps \dxdt\right|\\
&\leq C\norm{\phi}_{\Ccal^0}\eps^\beta\eps^\alpha\eps^{\alpha-1}\norm{\rho u}_{B_3^{\beta,\infty}}\norm{u}^2_{B_3^{\alpha,\infty}}+C\norm{\phi}_{\Ccal^1}\eps^\beta\eps^\alpha\norm{\rho}_{B_3^{\beta,\infty}}\norm{u}^2_{B_3^{\alpha,\infty}}.
\end{aligned}
\end{equation*}
These estimates show that $R_1^\eps$ and $R_2^\eps$ vanish as $\eps\to 0$.\\
To estimate terms $R_3^\eps$ and $R_4^\eps$ we integrate by parts and apply Lemma~\ref{CommutatorEstimates} to get the following
\begin{equation*}
    \begin{aligned}
    |R_3^\eps|&\leq\norm{\phi}_{\Ccal^1}\int_0^T\intT|p(\rho^\eps)-p^\eps(\rho)||u^\eps|\ \dxdt +\norm{\phi}_{\Ccal^0}\int_0^T\intT|p(\rho^\eps)-p^\eps(\rho)||\nabla u^\eps|\ \dxdt\\
    &\leq C\norm{p(\rho^\eps)-p^\eps(\rho)}_{L^{3/2}}(\norm{u^\eps}_{L^3}+\norm{\nabla u^\eps}_{L^3})\\
    &\leq C\left(\norm{\rho^\eps-\rho}_{L^3}^2+ \sup\limits_{y\in\supp\eta^\eps}\norm{\rho(\cdot)-\rho(\cdot-y)}_{L^3}^2
    \right)\left(1+\eps^{\alpha-1}\right)\norm{u}_{B_3^{\alpha,\infty}}\\
    &\leq C\left(\eps^{2\beta}\norm{\rho}_{B_3^{\beta,\infty}}^2+ \sup\limits_{|y|\leq\eps}|y|^{2\beta}\norm{\rho}_{B_3^{\beta,\infty}}^2
    \right)\left(1+\eps^{\alpha-1}\right)\norm{u}_{B_3^{\alpha,\infty}}\\
    &\leq C(\eps^{2\beta}+\eps^{2\beta+\alpha-1})\norm{u}_{B_3^{\alpha,\infty}}\norm{\rho}^2_{B_3^{\beta,\infty}}
    \end{aligned}
\end{equation*}
and similarly
\begin{equation*}
    \begin{aligned}
    |R_4^\eps|
    &\leq C\norm{S(\rho^\eps, \nabla\rho^\eps, \Delta\rho^\eps)-S^\eps(\rho,\nabla\rho, \Delta\rho)}_{L^{3/2}}(\norm{u^\eps}_{L^3}+\norm{\nabla u^\eps}_{L^3})\\
    &\leq C(\eps^{2\beta}+\eps^{2\beta+\alpha-1})\norm{u}_{B_3^{\alpha,\infty}}(\norm{\rho}^2_{B_3^{\beta,\infty}}+\norm{\nabla\rho}^2_{B_3^{\beta,\infty}}+\norm{\Delta\rho}^2_{B_3^{\beta,\infty}}).
    \end{aligned}
\end{equation*}
It now remains to estimate the last three commutator errors $R_5^\eps$, $R_6^\eps$ and $R_7^\eps$.
To this end we employ Lemma~\ref{lemma:compositiongrad} with function $f$ being $h'(\rho)$, $\kappa'(\rho)|\nabla\rho|^2$, and $\kappa(\rho)\Delta\rho$, respectively. We observe that by assumptions~\eqref{regassumption} and~\eqref{besovhypo} these functions belong to $L^\infty$.
Using again equality~\eqref{pointwisedecomp} we can estimate as follows.
\begin{equation*}
    \begin{aligned}
    |R_5^\eps| &\leq \int_0^T\intT |(\rho^\eps-\rho)(u^\eps-u)|(|h'(\rho^\eps)\nabla\phi|+|\phi\nabla h'(\rho^\eps)|)\ \dxdt\\
    &\leq \norm{\phi}_{\Ccal^1}\norm{\rho^\eps-\rho}_{L^3}\norm{u^\eps-u}_{L^3}\norm{h'(\rho^\eps)}_{L^3} + \norm{\phi}_{\Ccal^0}\norm{\rho^\eps-\rho}_{L^3}\norm{u^\eps-u}_{L^3}\norm{\nabla h'(\rho^\eps)}_{L^3}\\
    &\leq C\eps^\beta\eps^\alpha\norm{\rho}_{B_3^{\beta,\infty}}\norm{u}_{B_3^{\alpha,\infty}} + C\eps^\beta\eps^\alpha\eps^{\beta-1}\norm{\rho}^2_{B_3^{\beta,\infty}}\norm{u}_{B_3^{\alpha,\infty}},\\ \\
    |R_6^\eps|
    &\leq C (\norm{\kappa'(\rho^\eps)|\nabla\rho^\eps|^2}_{L^3}+\norm{\nabla(\kappa'(\rho^\eps)|\nabla\rho^\eps|^2)}_{L^3})\norm{\rho^\eps-\rho}_{L^3}\norm{u^\eps-u}_{L^3}\\
    &\leq C\eps^\beta\eps^\alpha\norm{\rho}_{B_3^{\beta,\infty}}\norm{u}_{B_3^{\alpha,\infty}} + C\eps^\beta\eps^\alpha\eps^{\beta-1}\norm{\rho}^2_{B_3^{\beta,\infty}}\norm{u}_{B_3^{\alpha,\infty}},
    \end{aligned}
\end{equation*}
and
\begin{equation*}
    \begin{aligned}
    |R_7^\eps|
    &\leq C (\norm{\kappa(\rho^\eps)\Delta\rho^\eps}_{L^3}+\norm{\nabla(\kappa(\rho^\eps)\Delta\rho^\eps)}_{L^3})\norm{\rho^\eps-\rho}_{L^3}\norm{u^\eps-u}_{L^3}\\
&\leq C\eps^\beta\eps^\alpha\norm{\rho}_{B_3^{\beta,\infty}}\norm{u}_{B_3^{\alpha,\infty}} + C\eps^\beta\eps^\alpha\eps^{\beta-1}\norm{\rho}^2_{B_3^{\beta,\infty}}\norm{u}_{B_3^{\alpha,\infty}}.
    \end{aligned}
\end{equation*}
For brevity the above calculations include only the first term coming from~\eqref{pointwisedecomp}, with the second term easily seen to produce estimates of the same order.\\
Thus the proof of the theorem is complete.

\subsection*{Acknowledgement}
This work was partially supported by the Simons - Foundation grant 346300 and the Polish Government MNiSW 2015-2019 matching fund;
AET thanks the Institute of Mathematics of the Polish Academy of Sciences, Warsaw, for their hospitality 
during his stay as a Simons Visiting Professor. P.G. and A.\'{S}-G. received support from the National Science Centre (Poland), 2015/18/M/ST1/00075. T.D acknowledges the support of the National
Science Centre (Poland), 2012/05/E/ST1/02218.


\begin{thebibliography}{10}

\bibitem{AntMarc2}
P.~Antonelli, P.~Marcati.
\newblock On the finite energy weak solutions to a system in quantum fluid
  dynamics.
\newblock {\em Comm. Math. Phys.}, {\bf 287}(2):657--686, 2009.

\bibitem{AntMarc}
P.~Antonelli, P.~Marcati.
\newblock The quantum hydrodynamics system in two space dimensions.
\newblock {\em Arch. Ration. Mech. Anal.}, {\bf 203}(2):499--527, 2012.


\bibitem{BarTiti}
C.~Bardos, E.~Titi.
\newblock Onsager's Conjecture for the Incompressible Euler Equations in Bounded Domains.
\newblock {\em Arch. Ration. Mech. Anal.}, 2017, DOI:https://doi.org/10.1007/s00205-017-1189-x.

\bibitem{BenShar}
C.~Bennett, R.~Sharpley.
\newblock {\em Interpolation of Operators}.
\newblock Pure and Applied Mathematics {\bf 129}, Academic Press Inc., Boston, 1988.

\bibitem{BucDeLSzV}
T.~Buckmaster, C.~De~Lellis, L.~Sz\'ekelyhidi, Jr., and V.~Vicol.
\newblock Onsager's conjecture for admissible weak solutions.
\newblock {\em arXiv}, (1701.08678), 2017.

\bibitem{Cafetal}
R.~E. Caflisch, I.~Klapper, and G.~Steele.
\newblock Remarks on singularities, dimension and energy dissipation for ideal
  hydrodynamics and {MHD}.
\newblock {\em Comm. Math. Phys.}, {\bf 184}(2):443--455, 1997.

\bibitem{BDD}
S.~Benzoni-Gavage, R.~Danchin, and S.~Descombes.
\newblock On the well-posedness for the euler-korteweg model in several space dimensions.
\newblock {\em Indiana Univ. Math. J.}, {\bf 56}:1499--1579, 2007.

\bibitem{BDDJ}
S.~Benzoni-Gavage, R.~Danchin, S.~Descombes, and D.~Jamet.
\newblock Structure of {K}orteweg models and stability of diffuse interfaces.
\newblock {\em Interfaces Free Bound.}, {\bf 7}(4):371--414, 2005.

\bibitem{Benzoni}
S.~Benzoni-Gavage.
\newblock Planar traveling waves in capillary fluids.
\newblock {\em Differential Integral Equations}, {\bf 26}(3-4):439--485, 2013.

\bibitem{cheskidov}
A.~Cheskidov, P.~Constantin, S.~Friedlander, and R.~Shvydkoy.
\newblock Energy conservation and {O}nsager's conjecture for the {E}uler
  equations.
\newblock {\em Nonlinearity}, {\bf 21}(6):1233--1252, 2008.

\bibitem{ConstETiti}
P.~Constantin, W.~E, and E.~S. Titi.
\newblock Onsager's conjecture on the energy conservation for solutions of
  {E}uler's equation.
\newblock {\em Comm. Math. Phys.}, {\bf 165}(1):207--209, 1994.

\bibitem{DLS09}
C.~De~Lellis and L.~Sz{\'e}kelyhidi, Jr.
\newblock The {E}uler equations as a differential inclusion.
\newblock {\em Ann. of Math. (2)}, {\bf 170}(3):1417--1436, 2009.

\bibitem{DLS10}
C.~De~Lellis and L.~Sz{\'e}kelyhidi, Jr.
\newblock On admissibility criteria for weak solutions of the Euler equations.
\newblock {\em Arch. Rational Mech. Anal.}, {\bf 195}:225--260, 2010.

\bibitem{DGSG}
T.~D\k{e}biec, P.~Gwiazda, A.~\'{S}wierczewska-Gwiazda.
\newblock A tribute to energy conservation for weak solutions.
\newblock {\em arXiv}, (1707.09794), 2017.

\bibitem{DonFeiMar}
D.~Donatelli, E.~Feireisl, P.~Marcati.
\newblock Well/ill posedness for the Euler-Korteweg-Poisson system and related problems.
\newblock {\em Comm. Partial Diff. Eq.}, {\bf 40}(7):1314--1335, 2015.

\bibitem{DrivasEyink}
T.~D. Drivas and G.~L. Eyink.
\newblock An onsager singularity theorem for turbulent solutions of
  compressible euler equations.
\newblock {\em to appear in Communications in Mathematical Physics}, 2017.

\bibitem{DuRo}
J.~Duchon and R.~Robert.
\newblock Inertial energy dissipation for weak solutions of incompressible
  {E}uler and {N}avier-{S}tokes equations.
\newblock {\em Nonlinearity}, {\bf 13}(1):249--255, 2000.

\bibitem{DunnSerrin}
J.~E.~Dunn, J.~Serrin.
\newblock On the thermomechanics of interstitial working.
\newblock {\em Arch. Ration. Mech. Anal.} {\bf 88}:95-133, 1985.

\bibitem{eyink}
G.~L. Eyink.
\newblock Energy dissipation without viscosity in ideal hydrodynamics. {I}.
  {F}ourier analysis and local energy transfer.
\newblock {\em Phys. D}, {\bf 78}(3-4):222--240, 1994.

\bibitem{FGSGW}
E.~Feireisl, P.~Gwiazda, A.~{\'{S}}wierczewska-Gwiazda, and E.~Wiedemann.
\newblock Regularity and {E}nergy {C}onservation for the {C}ompressible {E}uler
  {E}quations.
\newblock {\em Arch. Ration. Mech. Anal.}, {\bf 223}(3):1--21,
  2017.

\bibitem{GM}
I.~Gasser, P.~Markowich.
\newblock Quantum Hydrodynamics, Wigner transforms and the classical limit.
\newblock {\em Asymptotic Anal.}, {\bf 14}: 97-116, 1997

\bibitem{GT}
J.~Giesselmann, A.~Tzavaras.
\newblock Stability properties of the Euler-Korteweg system with nonmonotone pressures.
\newblock {\em Applicable Analysis} {\bf 96}(9):1528--1546, 2017

\bibitem{GLT}
J.~Giesselman, C.~Lattanzio, A.~ Tzavaras.
\newblock Relative energy for the korteweg theory and related hamiltonian flows in gas dynamics.
\newblock {\em Arch. Ration. Mech. Anal.}, {\bf 223}(3):1427--1484,
  2017.

\bibitem{GMSG}
P.~Gwiazda, M.~Mich\'{a}lek, A.~{\'{S}}wierczewska-Gwiazda.
\newblock A note on weak solutions of conservation laws and energy/entropy conservation.
\newblock {\em  arXiv}, (1706.10154), 2017

\bibitem{isett}
P.~Isett.
\newblock A {P}roof of {O}nsager's {C}onjecture.
\newblock {\em arXiv}, (1608.08301), 2016.

\bibitem{isett2}
P.~Isett.
\newblock On the Endpoint Regularity in {O}nsager's {C}onjecture.
\newblock {\em arXiv}, (1706.0154), 2017.

\bibitem{KangLee}
E.~Kang and J.~Lee.
\newblock Remarks on the magnetic helicity and energy conservation for ideal
  magneto-hydrodynamics.
\newblock {\em Nonlinearity}, {\bf 20}(11):2681--2689, 2007.

\bibitem{LeSh}
T.~M. Leslie and R.~Shvydkoy.
\newblock The energy balance relation for weak solutions of the
  density-dependent {N}avier-{S}tokes equations.
\newblock {\em J. Differential Equations}, {\bf 261}(6):3719--3733, 2016.

\bibitem{On1949}
L.~Onsager.
\newblock Statistical hydrodynamics.
\newblock {\em Nuovo Cimento (9)}, 6(Supplemento, 2 (Convegno Internazionale di
  Meccanica Statistica)):279--287, 1949.

\bibitem{scheffer}
V.~Scheffer.
\newblock An inviscid flow with compact support in space-time.
\newblock {\em J. Geom. Anal.}, {\bf 3}(4):343--401, 1993.

\bibitem{shnirel}
A.~Shnirelman.
\newblock Weak solutions with decreasing energy of incompressible {E}uler
  equations.
\newblock {\em Comm. Math. Phys.}, {\bf 210}(3):541--603, 2000.

\bibitem{shvydkoy}
R.~Shvydkoy.
\newblock On the energy of inviscid singular flows.
\newblock {\em J. Math. Anal. Appl.}, {\bf 349}:583--595, 2009.

\bibitem{Yu}
C.~Yu.
\newblock Energy conservation for the weak solutions of the compressible
  {N}avier--{S}tokes equations.
\newblock {\em Arch. Rational Mech. Anal.}, {\bf 225}(2):1073--1087, 2017.

\end{thebibliography}
\end{document}